\documentclass[11pt, oneside]{amsart}
\usepackage[margin=1in]{geometry}
\usepackage[T1]{fontenc}
\usepackage{graphicx}		
\usepackage{amssymb, amsthm, amsmath,amsfonts, hyperref}
\usepackage {tikz}
\usetikzlibrary {positioning}
\definecolor {processblue}{cmyk}{0.96,0,0,0}

\usepackage{todonotes}
\usepackage{comment}

\numberwithin{equation}{section}

\theoremstyle{plain}
\newtheorem{theorem}{Theorem}
\newtheorem{proposition}[theorem]{Proposition}
\newtheorem{corollary}[theorem]{Corollary}

\newtheorem{lemma}[theorem]{Lemma}

\theoremstyle{definition}
\newtheorem{remark}[theorem]{Remark}
\newtheorem{example}[theorem]{Example}

\newcommand{\CC}{\mathbb{C}}

\newcommand{\PP}{\mathbb{P}}
\newcommand{\ds}{\displaystyle}

\newcommand{\Mbar}{\overline{M}}
\newcommand{\ue}{{\mathrel{\cdot\mkern-3mu{-}\mkern-3mu\cdot}}}
\newcommand{\psinode}[2]{{\psi_{#1,#2}}}
\newcommand{\codim}{\operatorname{codim}}
\newcommand{\nodelabel}{\bullet}

\title[Products of boundary classes on $\Mbar_{0,n}$]{Products of boundary classes on $\Mbar_{0,n}$ via balanced weights}
\author{Maria Gillespie}
\address{Maria Gillespie, Department of Mathematics, Colorado State University, Fort Collins, CO, USA}
\email{\href{mailto:Maria.Gillespie@colostate.edu}{maria.gillespie@colostate.edu}}
\thanks{The first author was partially supported by NSF DMS award number 2054391.}

\author{Jake Levinson} 
\address{Jake Levinson, Département de mathématiques et de statistique, Université de Montréal, Montréal, QC, Canada} 
\email{\href{mailto:jake.levinson@umontreal.ca}{jake.levinson@umontreal.ca}} 
\thanks{The second author was partially supported by NSERC Discovery Grant RGPIN-2021-04169.} 

\begin{document}
\maketitle

\begin{abstract}
In this note, we give a simple closed formula for an arbitrary product, landing in dimension $0$, of boundary classes on the Deligne--Mumford moduli space $\Mbar_{0,n}$. For any such boundary strata $X_{T_1}, \ldots, X_{T_\ell}$, we show the intersection product $\int \prod_{i=1}^\ell [X_{T_i}]$ is either a signed product of multinomial coefficients, or zero, and provide a simple criterion for determining when it is nonzero.

We do not claim originality for our product formula, but to our knowledge it does not appear elsewhere in the literature.    
\end{abstract}

\section{Introduction}

Moduli spaces of curves, in particular the Deligne-Mumford-Knudsen moduli spaces $\overline{M}_{g, n}$ of stable curves, are of central interest in algebraic geometry. Of particular interest are enumerative calculations, and calculations in cohomology, involving loci of special curves and their classes.

In genus zero, the cohomology ring $H^*(\overline{M}_{0, n})$ is particularly straightforward to describe: Keel showed \cite{keel1992} that it is generated by the (Poincar\'e duals of) classes of \emph{boundary strata} $[X_T]$ indexed by certain trees, which stratify the moduli space according to the topological type of the curve. The intersection of such strata is always another stratum or empty, and when the intersection has the expected dimension, it is transverse in the sense of intersection multiplicities, that is, $[X_T] \cdot [X_{T'}] = [X_T \cap X_{T'}]$.

Keel's presentation includes additional linear relations which allow one to reduce to the transverse case in computations. This leaves implicit the question of the degrees of intersection products of non-transverse strata, such as self-intersection numbers. In fact there is a simple explicit formula (Proposition \ref{prop:balanced-weight-formula}), which we call the ``balanced weights formula'', for an arbitrary product of boundary classes that lands in dimension $0$ in $H^\ast(\Mbar_{0,n})$.

\begin{example}
Writing $D_I \subseteq \overline{M}_{0, n}$ for the divisor on which the marked points $\{i \in I\}$ have collided, we have on $\overline{M}_{0, 15}$,
\[[D_{12}]^2 [D_{345}]^3 [D_{12345678}]^4 [D_{11,12}] [D_{13,14,15}]^2 = (-1)^7 \binom{3}{1,2}\binom{2}{1,1}\binom{3}{1,1,1}[pt] = -36 [pt],\]
where $[pt]$ denotes the cohomology class of a point and $\binom{n}{k_1, \ldots, k_r}$ denotes a multinomial coefficient. This calculation is carried out in Example \ref{exa:balanced-weight}: two of the multinomial coefficients come from edges in a tree, one comes from a vertex.
\end{example}

Although informally known to experts, to our knowledge this formula does not appear explicitly in the literature (a close formula is \cite{yang2010calculating}, Equation 14), and we provide it here for ease of reference along with clear exposition. In this note, we revisit intersections of boundary strata, then state and prove the balanced weights formula.

\section{Setup and nonempty intersections}

Two thorough introductions to this material are in  \cite{cavalieri2016, k:psi}. In general a stable genus $0$ curve consists of a number of $\PP^1$'s joined at nodes in a tree structure, along with $n$ marked points such that each $\PP^1$ contains at least $3$ total nodes and marked points.  Its \textit{dual tree} is the graph whose vertices are the $\PP^1$ components and the marked points, and whose edges are given by incidence.  (See Figure \ref{fig:dual-tree}.)  The dual tree is a leaf-labeled \textit{at least trivalent} tree, meaning that every internal (non-leaf) vertex has degree at least $3$.  Such a tree is called a \textit{stable tree}.

\begin{figure}
    \centering
    \includegraphics[width=15cm]{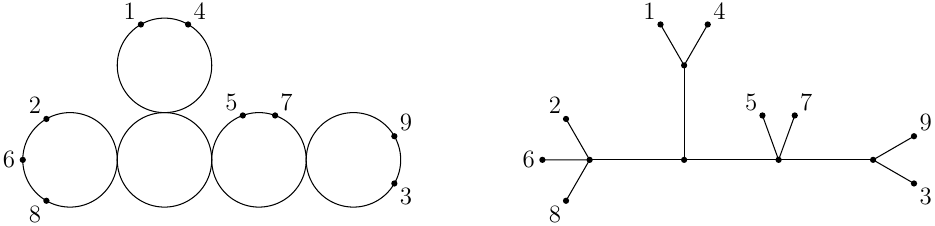}
    \caption{A stable curve in $\Mbar_{0,9}$ and its dual tree.  While we assume we are working over $\CC$, we will draw our figures so that each $\PP^1$ is a circle, representing $\PP^1_\mathbb{R}$.}
    \label{fig:dual-tree}
\end{figure}

We write $\Mbar_{0,S}$ for the moduli space of stable curves of genus $0$ with marked points labeled by the finite set $S$. When $S = \{1, \ldots, n\}$, we write $\Mbar_{0,n}$. For a decomposition $S = A \sqcup B$, we write $D_{A|B}$ for the divisor consisting of stable curves $C$ with a node $q \in C$ separating the marked points in $A$ from those in $B$.

More generally, let $T$ be a stable tree with leaves labeled by $S$. We write $X_T$ for the corresponding closed boundary stratum of $\Mbar_{0,S}$, the closure of the locus of curves with dual tree $T$. Let $v(T)$ and $e(T)$ denote the sets of internal (non-leaf) vertices of $T$ and internal (non-leaf) edges. We have
\begin{equation} \label{eq:stratum-isom-product}
X_T \cong \prod_{v \in v(T)} \Mbar_{0,\deg(v)}.
\end{equation}
By abuse of notation, we write $\Mbar_{0,v}$ for the factor corresponding to $v$. If $e \in e(T)$ is an internal edge, then $T \setminus e$ has two connected components, dividing the marked points into two sets $A_e, B_e$. Write $D_e := D_{A_e|B_e}$. Then
\begin{equation}\label{eqn:complete_intersection}
X_T = \bigcap_{e \in T} D_e, \qquad [X_T] = \prod_{e \in e(T)} [D_e] \in H^*(\Mbar_{0, n}).
\end{equation}
In particular, $X_T$ is a complete intersection and $\codim(X_T) = |e(T)|$. 

\begin{example}\label{ex:tree}
    If $T$ is the tree in Figure \ref{fig:dual-tree}, we have \[X_T=D_{268|134579}\cap D_{12468|3579}\cap D_{14|2356789}\cap D_{39|1245678}.\] 
\end{example}

Conversely, every intersection of boundary divisors is either a boundary stratum $X_T$, or empty. In particular, $X_T \cap D_{A|B}$ is nonempty if and only if there exists $v \in v(T)$ such that the components of $T \setminus v$ form a set partition of $[n]$ that refines $A \sqcup B$.  In this case we say $X_T$ and $D_{A|B}$ are \textit{compatible}. Note that two divisors $D_{A|B}$ and $D_{X|Y}$ are compatible if either $A\subseteq X$, $A\subseteq Y$, $A\supseteq X$, or $A\supseteq Y$.  (Up to interchanging the partitions or the blocks, these four conditions are equivalent.) 

In fact more is true:

\begin{proposition}[\cite{phylogenetics}, Theorem 3.1.4] \label{prop:flag}
   An intersection $X_T\cap X_{T'}$ is nonempty if and only if for every pair of internal edges $e\in T$ and $e'\in T'$, $D_e \cap D_{e'}$ is nonempty.
\end{proposition}

Equivalently, the \textit{boundary complex} of $\overline{M}_{0, n}$ is a \textit{flag complex}.  That is, the simplicial complex whose vertices are the boundary divisors $D_{A|B}$ and whose faces are the strata is \textit{$2$-determined}: a face is in the complex if and only if all of its edges are. This rather special property does not hold for most stratified spaces, such as $\mathbb{P}^n$ with its coordinate hyperplanes. See \cite{giansiracusa} for additional discussion of the `folklore' nature of Proposition \ref{prop:flag} and its usage in phylogenetics. 
For $g > 0$, it is shown in \cite{clader2020boundary} that the boundary complex of $\overline{M}_{g, n}$ is flag if and only if $g \leq 1$ or $n \leq 1$ or $g = n = 2$.

In fact, there is a simple proof in the $g=0$ setting using a graph coloring, which appears in \cite{phylogenetics} and as a special case of the work in \cite{clader2020boundary}, and which we provide a slightly different version of here for completeness. The proof reduces to the following ``edge condition'' for compatibility of $X_T$ with a boundary divisor $D$.  

\begin{lemma}\label{lem:tree-times-divisor}
 The intersection $X_T\cap D_{X|Y}$ is nonempty if and only if, for every internal edge $e\in T$, $D_e$ and $D_{X|Y}$ are compatible.  
 
 Moreover, if it is nonempty, $X_T\cap D_{X|Y}=X_S$ where either $S=T$ or $S$ is formed by separating the branches of $T$ having leaves in $X$ from those in $Y$ by an edge.
\end{lemma}

\begin{proof}
If any $D_e$ is not compatible with $D_{X|Y}$, their intersection is empty, and therefore the entire intersection is empty.   For the other direction, for any internal edge $e\in T$, let $D_e$ correspond to partition $A|B$, with $A,B$ ordered so that either $A\subseteq X$ or $A\supseteq X$.  Then we color edge $e$ blue if $A\subseteq X$ and red if $A\supseteq X$.  Also color the leaf edges blue if the corresponding leaf is in $X$ and red if it is in $Y$. (See Example \ref{ex:colors}.)

Let $p$ be any path in the tree connecting a blue leaf to a red leaf, and define $v$ to be the (internal) vertex preceding the first red edge in $p$. We claim that each branch $\mathcal{B}$ at $v$ is monochromatic.

Choose any branch $\mathcal{B}$ at $v$. Let $e_v$ be the unique edge of $\mathcal{B}$ attached to $v$.  Suppose $e_v$ is blue. Then either the set of leaves $A$ of branch $\mathcal{B}$ satisfies $A\subseteq X$, or the complement $B$ of $A$ has $B\subseteq X$, by the definition of the blue edge coloring.  We show that in fact $A\subseteq X$. This will imply $\mathcal{B}$ is monochromatic (blue): for any other edge $f$ on $\mathcal{B}$, the leaves on the side of $f$ opposite $v$ form a subset $A_f\subseteq A\subseteq X$, so $f$ is also colored blue.


Assume for contradiction that $B\subseteq X$, and consider the red edge $e_v'$ along path $p$ connecting to $v$, and let $w$ be the other vertex of $e_v'$.  The branch $\mathcal{B}'$ at $w$ containing edge $e_v'$ contains $\mathcal{B}$ and at least one more leaf than $A$, because vertex $v$ is at least trivalent.  If $A'$ and $B'$ are the sets of leaves on the sides of $e_v'$ corresponding to $v$ and $w$ respectively, then we have $A'\supset A$ and so $B'\subseteq B\subseteq X$.  But then edge $e'_v$ should be colored blue (since $B'\subseteq X$), a contradiction. 


Finally, suppose instead that edge $e_v$ is red.  Then $A\supseteq X$ and $B\subseteq Y$, and an identical proof as the above shows that the branch $\mathcal{B}$ is monochromatic (red).  This proves the claim.

To finish the proof, we note that if exactly one branch at $v$ is a different color than the others, then $X_T\cap D_{X|Y}=X_T$.  Otherwise, we construct a new tree $S$ by creating a new edge $\varepsilon$ and attaching the blue branches at its left vertex and the red branches at its right, and we have $X_{S}=X_T\cap D_{X|Y}$. 
\end{proof}

\begin{example}\label{ex:colors}
    We compute $X_T\cap D_{124568|379}$ where $T$ is as in Figure \ref{fig:dual-tree}.  Note the coloring of the branches as in the proof above; the resulting tree $S$ is at right.
    \begin{center}
        \includegraphics{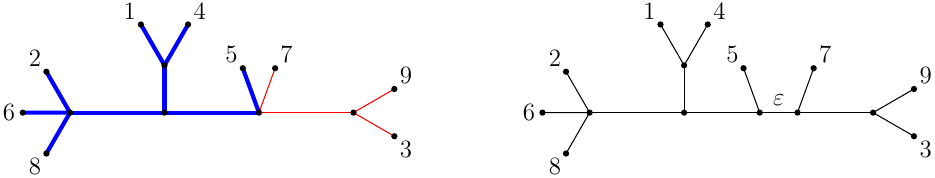}
    \end{center}
\end{example}

Proposition \ref{prop:flag} now follows from Lemma \ref{lem:tree-times-divisor}, by writing a stratum as a product of divisors  as in Equation \eqref{eqn:complete_intersection}.  We also have the following corollary.

\begin{corollary}
For any collection of strata $\{ X_{T_i} \}$, we have $\bigcap X_{T_i} \ne \varnothing$ (and is equal to a boundary stratum $X_S$) if and only if, for all $i, j$, $X_{T_i} \cap X_{T_j} \ne \varnothing$.
\end{corollary}

\subsection{Psi classes and intersection products}

The results below can be found in \cite[Section 2.5.2]{cavalieri2016}.  For $i \in S$, we write $\psi_i \in H^*(\Mbar_{0,S})$ for the \emph{psi class} of the marked point $i$, corresponding to the $i$-th cotangent line bundle $\mathbb{L}_i$. 

For a boundary stratum $X_T$, each edge $e = v \ue v' \in e(T)$ corresponds to two `local' or `half-edge' psi classes: we write $\psinode{v}{e} \in H^*(\Mbar_{0,v})$ and $\psinode{v'}{e} \in H^*(\Mbar_{0,v'})$ for the psi classes associated to $e$ from the $v$ factor and $v'$ factor, marking the node of the curve corresponding to $e$. We identify these classes with their pullbacks in $H^*(X_T)$ along the projections $X_T \to \Mbar_{0,v}$ and $X_T \to \Mbar_{0,v'}$ of Equation \eqref{eq:stratum-isom-product}. 

For a tuple $(k_1, \ldots, k_n)$ such that $\sum k_i = n-3$, we have, by the \textit{string equation},
\begin{equation}\label{eqn:multinomial}
\int_{\Mbar_{0,n}}(\psi_1^{k_1} \cdots \psi_n^{k_n}) = \binom{n-3}{k_1, \ldots, k_n},
\end{equation}
the multinomial coefficient.

Next, let $D = D_{A|B}$ be a boundary divisor. We have $D \cong \Mbar_{0,A \cup \nodelabel} \times \Mbar_{0,B \cup \nodelabel}$, where the bullet $\nodelabel$ refers to the marking for the node. Let $\iota : D \hookrightarrow \Mbar_{0,n}$ be the inclusion. Then
\begin{equation}\label{eqn:D^2}
[D]^2 = -\iota_*(\psinode{A}{\nodelabel} + \psinode{B}{\nodelabel}), 
\quad \text{ or equivalently } \quad
\iota^*[D] = -(\psinode{A}{\nodelabel}+\psinode{B}{\nodelabel})
\end{equation}
where in the formulas above, $\psi_{A,\bullet}$ and $\psi_{B,\bullet}$ denote the `local' (half-edge) psi classes defined above. Similarly, for a boundary stratum $\iota : X_T \hookrightarrow \Mbar_{0,S}$ and an edge $e = v \ue v' \in e(T)$, we have
\begin{equation}\label{eqn:X_T*D_e}\iota^*[D_e] = -(\psinode{v}{e} + \psinode{v'}{e}).
\end{equation}

\section{Self-intersections and a balanced weight formula}

The balanced weights formula is most naturally phrased in terms of products of divisors (Proposition \ref{prop:balanced-weight-formula}). The general case follows since every boundary class is a product of divisor classes (see Remark \ref{rem:products-of-strata}).

We fix a set of divisors whose intersection is nonempty, hence equal to $X_T$ for some stable tree $T$. We write $D_e$ for the divisor corresponding to the edge $e \in e(T)$. Let $\iota : X_T \hookrightarrow \overline{M}_{0,n}$ be the inclusion.

For each $v \in v(T)$, we let $n(v) := \deg(v) - 3$, the dimension of the corresponding $\Mbar_{0,v}$. We fix exponents $k(e)$ for each $e \in e(T)$, and we examine the product
\begin{equation}\label{eq:divisor-product}
\prod_{e \in e(T)} [D_e]^{1+k(e)} = [X_T] \cdot \prod_{e \in e(T)} [D_e]^{k(e)}.
\end{equation}
The condition that this be in dimension $0$ corresponds numerically to the equalities
\begin{equation}\label{eq:dimension-assumption}
\sum_{v \in v(T)} n(v) = \dim X_T = \sum_{e \in e(T)} k(e).
\end{equation}



\begin{lemma}\label{lem:pullback}
We have $\ds{\prod_{e \in e(T)} [D_e]^{1+k(e)} = \iota_*\bigg(
\prod_{e \in e(T)}(-1)^{k(e)}(\psinode{v}{e}+\psinode{v'}{e})^{k(e)}
\bigg)}$.
\end{lemma}

\begin{proof}
The first occurrence of each distinct factor gives $[X_T]$. Then by the projection formula (see e.g. \cite[Section 3.3]{Hatcher}) our product has the form
\[
[X_T] \cdot \prod_{e \in e(T)} {D_e}^{k(e)}
= \iota_* \iota^* \bigg( \prod_{e \in e(T)} {D_e}^{k(e)} \bigg).
\]
Combining with the self-intersection formula \eqref{eqn:X_T*D_e} gives the desired form.
\end{proof}

The data $n(v)$, $k(e)$ essentially decorates the tree $T$ with its leaves contracted. Working in the cohomology ring $H^*(X_T)$ of the stratum, we label the edge $e$ with a factor $(-1)^{k(e)}(\psinode{v}{e}+\psinode{v'}{e})^{k(e)}$. See Fig. \ref{fig:decorated_tree}, ``Left'' and ``Middle''.

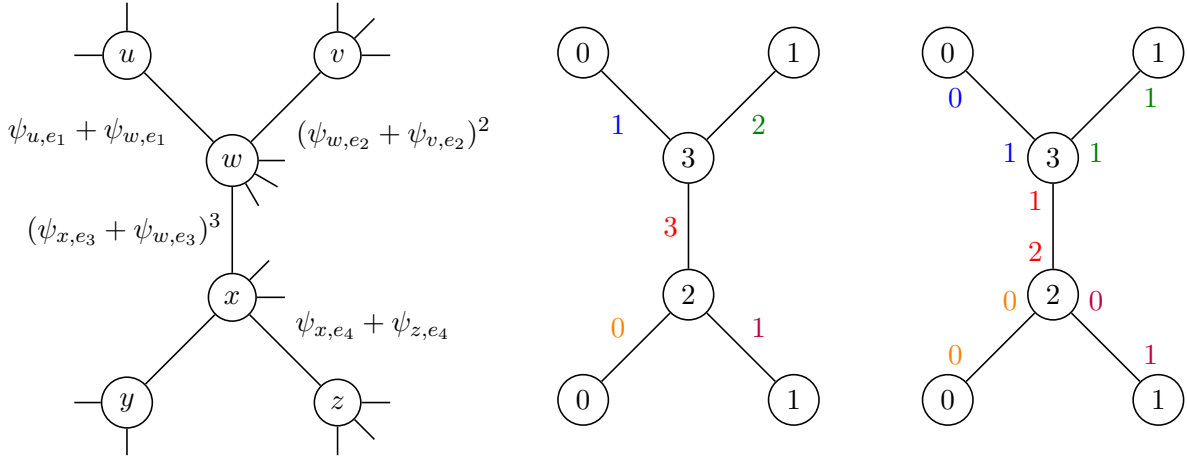
\begin{figure}
\centering
\ \hspace{-1.2cm}
\raisebox{-.4cm}{
\begin {tikzpicture}[semithick, state/.style = {circle, top color = white, draw, text=black},x=1.4cm,y=1.4cm]
\node[state] at (0, 0) (v1) {$y$};
\node[state] at (2, 0) (v2) {$z$};
\node[state] at (1, 1) (v3) {$x$};
\node[state] at (1, 2.3) (v4) {$w$};
\node[state] at (0, 3.3) (v5) {$u$};
\node[state] at (2, 3.3) (v6) {$v$};
\path (v3) edge node[above left] {} (v1);
\path (v3) edge node[above right] {$\psinode{x}{e_4}+\psinode{z}{e_4}$} (v2);
\path (v3) edge node[left] {$(\psinode{x}{e_3}+\psinode{w}{e_3})^3$} (v4);
\path (v4) edge node[below left] {$\psinode{u}{e_1}+\psinode{w}{e_1}$} (v5);
\path (v4) edge node[below right] {$(\psinode{w}{e_2}+\psinode{v}{e_2})^2$} (v6);
\path (v1) edge +(-0.5, 0);
\path (v1) edge +(0, -0.5);
\path (v2) edge +(0.5, 0);
\path (v2) edge +(0, -0.5);
\path (v2) edge +(0.35, -0.35);
\path (v3) edge +(0.35, 0.35);
\path (v3) edge +(0.5, 0);
\path (v4) edge +(0.5, 0);
\path (v4) edge +(0.43, -0.25);
\path (v4) edge +(0.25, -0.43);
\path (v5) edge +(-0.5, 0);
\path (v5) edge +(0, 0.5);
\path (v6) edge +(0.5, 0);
\path (v6) edge +(0.35, 0.35);
\path (v6) edge +(0, 0.5);
\end{tikzpicture}}
\hspace{0.5cm}
\begin {tikzpicture}[semithick, state/.style = {circle, top color = white, draw, text=black},x=1.4cm,y=1.4cm]
\node[state] at (0, 0) (v1) {$0$};
\node[state] at (2, 0) (v2) {$1$};
\node[state] at (1, 1) (v3) {$2$};
\node[state] at (1, 2.3) (v4) {$3$};
\node[state] at (0, 3.3) (v5) {$0$};
\node[state] at (2, 3.3) (v6) {$1$};
\path (v3) edge node[above left] {\color{orange}$0$} (v1);
\path (v3) edge node[above right] {\color{purple}$1$} (v2);
\path (v3) edge node[left] {\color{red}$3$} (v4);
\path (v4) edge node[below left] {\color{blue}$1$} (v5);
\path (v4) edge node[below right] {\color{green!50!black}$2$} (v6);
\end{tikzpicture}
\hspace{1.1cm}
\begin {tikzpicture}[semithick, state/.style = {circle, top color = white, draw, text=black},x=1.4cm,y=1.4cm]
\node[state] at (0, 0) (v1) {$0$};
\node[state] at (2, 0) (v2) {$1$};
\node[state] at (1, 1) (v3) {$2$};
\node[state] at (1, 2.3) (v4) {$3$};
\node[state] at (0, 3.3) (v5) {$0$};
\node[state] at (2, 3.3) (v6) {$1$};
\path (v3) edge node[above left, pos=0.1] {\color{orange}$0$} node[above left, pos=0.9] {\color{orange}$0$} (v1);
\path (v3) edge node[above right, pos=0.1] {\color{purple}$0$} node[above right, pos=0.9] {\color{purple}$1$} (v2);
\path (v3) edge node[left, pos=0.2] {\color{red} $2$} node[left, pos=0.8] {\color{red}$1$} (v4);
\path (v4) edge node[below left, pos=0.1] {\color{blue}$1$} node[below left, pos=0.9] {\color{blue}$0$} (v5);
\path (v4) edge node[below right, pos=0.1] {\color{green!50!black}$1$} node[below right, pos=0.9] {\color{green!50!black}$1$} (v6);
\end{tikzpicture}
    \caption{A decorated tree with leaf labels omitted. {\bf Left}: Factors of $(\psinode{v}{e} + \psinode{v'}{e})^{k(e)}$ are shown. {\bf Middle}: Simplified decoration, showing only vertex weights (dimensions) and edge weights (exponents). {\bf Right:} The unique balanced weights, giving $-36$ times the class of a point (see Ex. \ref{exa:balanced-weight}).
    }
    \label{fig:decorated_tree}
\end{figure}

We continue directly in $H^*(X_T)$. Expanding out the product $\prod_e (\psinode{v}{e}+\psinode{v'}{e})^{k(e)}$ from Lemma~\ref{lem:pullback} gives a sum of monomials in the psi classes, but many such monomials vanish: for each $v$, by dimensionality we cannot take more than $n(v)$ factors of the form $\psinode{v}{e}$. In fact, Equation \ref{eq:dimension-assumption} implies \emph{at most one} monomial contributes, as the following combinatorial lemma shows.
\begin{lemma}\label{lem:balanced-weights}
Let $T$ be a stable tree with vertex and edge weights $n : v(T) \to \mathbb{N}$ and $k : e(T) \to \mathbb{N}$, such that $\sum_{e \in e(T)} k(e) = \sum_{v \in v(T)} n(v)$.

There is at most one choice of decomposition of each edge weight $k(e)$ nonnegatively as $k(v,e) + k(v', e)$ (where $e = v\ue v'$), such that, for all $v \in v(T)$, it holds $n(v) = \sum_{e \text{ incident to } v} k(v, e)$.
\end{lemma}
\begin{proof}
Proceed greedily: delete all leaves from $T$. Pick any $v \in v(T)$ that is now a leaf, and let $e = v \ue v'$ be the unique edge incident to $v$. Then $k(e)$ must decompose as $\big(k(v,e), k(v',e)\big) = \big(n(v), k(e) - n(v)\big)$. If $n(v) \leq k(e) \leq n(v) + n(v')$, we proceed inductively, deleting $v$ and lowering the weight of $v'$. If not, no valid decomposition exists.
\end{proof}
When it exists, we call the unique weighting $k : \{(v,e) : e \text{ incident to } v\} \to \mathbb{N}$ the {\bf balanced weights for} $T, n, k$. 

\begin{proposition}[Balanced weights formula] \label{prop:balanced-weight-formula}
If balanced weights for $T, n, k$ do not exist, then $\prod_e [D_e]^{1+k(e)} = 0$. Otherwise, we have
$$
\int_{\Mbar_{0,n}} \prod_{e \in e(T)} [D_e]^{1+k(e)} = 
(-1)^{\dim X_T}
\cdot
\prod_{\substack{e \in T \\ e = v \ue v'}} \binom{k(e)}{k(v,e), k(v',e)}
\cdot
\prod_{v \in v(T)} \binom{n(v)}{k(v,e_1), k(v, e_2), \ldots}
$$
where the multinomial coefficient for $v$ ranges over the edges incident to $v$.
\end{proposition}
\begin{proof}
Continuing from Lemma \ref{lem:pullback}, note first that the global sign is 
\[
(-1)^{\sum_e k(e)} = (-1)^{\sum_v n(v)} = (-1)^{\dim X_T}.
\]
Next, expand out all the products $(\psinode{v}{e}+\psinode{v'}{e})^{k(e)}$ and collect factors according to their vertex labels. Any monomial in the $\psi$ classes involving more than $n(v)$ factors for any $v$, is zero for dimension reasons. In the remaining monomials, each $v$ receives $\leq n(v)$ factors, hence exactly $n(v)$ since $\sum_e k(e) = \sum_v n(v)$. That is, these monomials come from taking, from each $(\psinode{v}{e} + \psinode{v'}{e})^{k(e)}$, the term $\psinode{v}{e}^{k(v,e)} \psinode{v'}{e}^{k(v',e)}$, using balanced weights $(k(v, e) : e \text{ incident to } v)$. 

Assuming such weights exist, they are unique by Lemma \ref{lem:balanced-weights}. We then have
\begin{align}
\prod_{e \in e(T)}
(\psinode{v}{e}+\psinode{v'}{e})^{k(e)}
&=
\prod_{e \in e(T)}\binom{k(e)}{k(v, e), k(v', e)} 
\cdot \psinode{v}{e}^{k(v, e)} \psinode{v'}{e}^{k(v', e)} \\
&=
\prod_{e \in e(T)}\binom{k(e)}{k(v, e), k(v', e)} 
\cdot
\prod_{v \in v(T)} \prod_{e \text{ incident to } v} \psinode{v}{e}^{k(v, e)}.
\end{align}
Since the weights are balanced, $\sum_{e \text{ incident to } v} k(e, v) = n(v)$. Thus the product of psi classes on $\Mbar_{0,v}$ is the multinomial coefficient by Equation \eqref{eqn:multinomial}. This completes the proof.
\end{proof}

We can use the formula $\binom{n}{k_1,\ldots,k_r}=\frac{n!}{k_1!\cdots k_r!}$ to express Proposition \ref{prop:balanced-weight-formula} as of a ratio of factorials:

\begin{corollary}[Balanced weights ratio]\label{cor:ratio}
With notation as above, we have
\[\int_{\Mbar_{0,n}}\prod_{e \in e(T)} [D_e]^{1+k(e)}=(-1)^{\dim X_T}
\frac{\vphantom{\big|}\prod_v n(v)!\cdot \prod_e k(e)!}{\vphantom{\big|}\prod_{v,e} k(v,e)!^2}.\]
\end{corollary}

\begin{example} \label{exa:balanced-weight}
For simplicity we write $D_A$ in place of $D_{A|B}$. On $\Mbar_{0, 15}$, consider the product
\[
D_{12}^2 D_{345}^3 D_{12345678}^4 D_{11,12} D_{13,14,15}^2.
\]
The resulting set-theoretic intersection gives a tree $T$ with $\dim(X_T) = 7$ and $6$ internal vertices $\{u, v, w, x, y, z\}$, of the combinatorial form shown in Fig. \ref{fig:decorated_tree}. In the figure, the leaf labels can be taken to be $1, 2, 3, \ldots, 15$ from left to right and top to bottom.

The resulting balanced weights are also shown. The $w\ue x$ edge contributes a factor of $\binom{3}{1,2} = 3$, the $w \ue v$ edge contributes a factor of $\binom{2}{1,1} = 2$, and  the vertex $w$ contributes a factor of $\binom{3}{1,1,1} = 6$. All other edge and vertex factors are $1$. The global sign is $(-1)^{\dim X_T} = -1$. By Proposition \ref{prop:balanced-weight-formula}, we conclude
\begin{align*}
D_{12}^2 D_{345}^3 D_{12345678}^4 D_{11,12} D_{13,14,15}^2
=
X_T \cdot D_{12} D_{345}^2 D_{12345678}^3 D_{13,14,15}
=
-36 [\mathrm{pt}].
\end{align*}
We could also have used the formula in Corollary \ref{cor:ratio}, which in this case is $$(-1)^7 \frac{(3!\cdot 2!\cdot 1!\cdot 1!\cdot 0!\cdot 0!)(3!\cdot 2!\cdot 1!\cdot 1!\cdot 0!)}{(2!\cdot 1!\cdot 1!\cdot 1!\cdot 1!\cdot 0!\cdot 0!\cdot 0!\cdot 0!)^2}=\frac{-6\cdot 2 \cdot 6 \cdot 2}{2^2}=-36.$$

\end{example}

\begin{remark}[Products of arbitrary strata] \label{rem:products-of-strata}
Let $T_1, \ldots, T_\ell$ be stable trees such that $\ds{\bigcap X_{T_i}}$ is nonempty, hence equal to a stratum $X_T$ for some $T$. Assume $\ds{\sum \codim(X_{T_i}) = \dim \Mbar_{0,n}}$. The intersection product $\ds{\prod [X_{T_i}]}$ is then a class in dimension zero, and is computed by Proposition~\ref{prop:balanced-weight-formula}. Explicitly, for each edge $e \in e(T)$, we let
\[
k(e) := \#\{ i : X_{T_i} \subseteq D_e \} - 1,
\]
that is, $k(e)$ counts how many times the boundary divisor $D_e$ occurs after its first occurrence, when each $X_{T_i}$ is expressed as a complete intersection. Equivalently, if $D_e = D_{A|B}$, then $k(e)$ counts the $T_i$ having an edge dividing the marked points $A$ from the marked points $B$ (after the first occurrence). The vertex weights $n(v)$ and balanced weights for $T, n, k$ are formed as above.
\end{remark}

\begin{remark}[Psi class factors]
The formula extends to products of boundary classes and psi classes, by including $\psi_i$ factors in the balanced weights calculation on $X_T$. The result is again a signed product of multinomial coefficients, or zero.

For instance, starting with the same tree as in the example above, suppose we wish to compute the product $$\psi_4^2\psi_7\cdot D_{12}^2 D_{345} D_{12345678}^3 D_{11,12} D_{13,14,15}^2
=
X_T \cdot D_{12} D_{12345678}^2 D_{13,14,15}\cdot \psi_4\psi_7^2.$$  This modifies the above example by adding edges labeled with the exponents $1$ and $2$ to the leaf edges corresponding to marked points $4$ and $7$, respectively.  The same algorithm then produces at most one splitting of the other edge labels, as shown in Figure \ref{fig:psi_tree}.  The corresponding product of multinomial coefficients is then $$\binom{1}{0,1}\binom{0}{0,0}\binom{2}{2,0}\binom{0}{0,0}\binom{1}{0,1}\cdot \binom{3}{1,2,0,0}\binom{0}{0}\binom{1}{0,1}\binom{2}{2,0,0}\binom{1}{1}\binom{0}{0}=3,$$ and since the number of divisors involved is even, we have the positive intersection product $3[pt]$.
\end{remark}

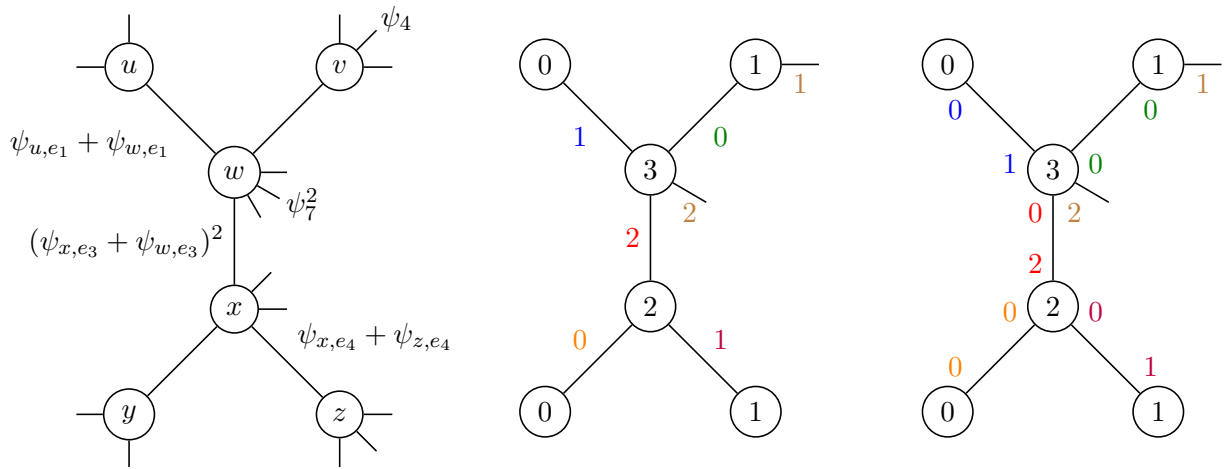
\begin{figure}
\centering
\ \hspace{-1.2cm}
\raisebox{-.4cm}{
\begin {tikzpicture}[semithick, state/.style = {circle, top color = white, draw, text=black},x=1.4cm,y=1.4cm]
\node[state] at (0, 0) (v1) {$y$};
\node[state] at (2, 0) (v2) {$z$};
\node[state] at (1, 1) (v3) {$x$};
\node[state] at (1, 2.3) (v4) {$w$};
\node[state] at (0, 3.3) (v5) {$u$};
\node[state] at (2, 3.3) (v6) {$v$};
\path (v3) edge node[above left] {} (v1);
\path (v3) edge node[above right] {$\psinode{x}{e_4}+\psinode{z}{e_4}$} (v2);
\path (v3) edge node[left] {$(\psinode{x}{e_3}+\psinode{w}{e_3})^2$} (v4);
\path (v4) edge node[below left] {$\psinode{u}{e_1}+\psinode{w}{e_1}$} (v5);
\path (v4) edge node[below right] {} (v6);
\path (v1) edge +(-0.5, 0);
\path (v1) edge +(0, -0.5);
\path (v2) edge +(0.5, 0);
\path (v2) edge +(0, -0.5);
\path (v2) edge +(0.35, -0.35);
\path (v3) edge +(0.35, 0.35);
\path (v3) edge +(0.5, 0);
\path (v4) edge +(0.5, 0);
\path (v4) edge +(0.43, -0.25);
\path (v4) edge +(0.25, -0.43);
\path (v5) edge +(-0.5, 0);
\path (v5) edge +(0, 0.5);
\path (v6) edge +(0.5, 0);
\path (v6) edge +(0.35, 0.35);
\path (v6) edge +(0, 0.5);
\node[right, outer sep=2pt] at (1.35, 2) {$\psi_7 ^2$};
\node[above right, outer sep=2pt] at (2.25, 3.5) {$\psi_4$};
\end{tikzpicture}}
\hspace{0.5cm}
\begin {tikzpicture}[semithick, state/.style = {circle, top color = white, draw, text=black},x=1.4cm,y=1.4cm]
\node[state] at (0, 0) (v1) {$0$};
\node[state] at (2, 0) (v2) {$1$};
\node[state] at (1, 1) (v3) {$2$};
\node[state] at (1, 2.3) (v4) {$3$};
\node[state] at (0, 3.3) (v5) {$0$};
\node[state] at (2, 3.3) (v6) {$1$};
\path (v3) edge node[above left] {\color{orange}$0$} (v1);
\path (v3) edge node[above right] {\color{purple}$1$} (v2);
\path (v3) edge node[left] {\color{red}$2$} (v4);
\path (v4) edge node[below left] {\color{blue}$1$} (v5);
\path (v4) edge node[below right] {\color{green!50!black}$0$} (v6);

\path (v4) edge node[below] {\color{brown} $2$} +(0.53, -0.31);

\path (v6) edge node[below] {\color{brown} $1$} +(0.6, 0);
\end{tikzpicture}
\hspace{1.1cm}
\begin {tikzpicture}[semithick, state/.style = {circle, top color = white, draw, text=black},x=1.4cm,y=1.4cm]
\node[state] at (0, 0) (v1) {$0$};
\node[state] at (2, 0) (v2) {$1$};
\node[state] at (1, 1) (v3) {$2$};
\node[state] at (1, 2.3) (v4) {$3$};
\node[state] at (0, 3.3) (v5) {$0$};
\node[state] at (2, 3.3) (v6) {$1$};
\path (v3) edge node[above left, pos=0.1] {\color{orange}$0$} node[above left, pos=0.9] {\color{orange}$0$} (v1);
\path (v3) edge node[above right, pos=0.1] {\color{purple}$0$} node[above right, pos=0.9] {\color{purple}$1$} (v2);
\path (v3) edge node[left, pos=0.2] {\color{red} $2$} node[left, pos=0.8] {\color{red}$0$} (v4);
\path (v4) edge node[below left, pos=0.1] {\color{blue}$1$} node[below left, pos=0.9] {\color{blue}$0$} (v5);
\path (v4) edge node[below right, pos=0.1] {\color{green!50!black}$0$} node[below right, pos=0.9] {\color{green!50!black}$0$} (v6);

\path (v4) edge node[below left] {\color{brown} $2$} +(0.53, -0.31);

\path (v6) edge node[below] {\color{brown} $1$} +(0.6, 0);
\end{tikzpicture}
    \caption{Computing a mixed product of $\psi$ classes and divisors. The algorithm for finding the balanced weights is the same, where we simply treat the $\psi_4^2$ as weighing the leaf edge connected to leaf number $4$ as having weight $2$, and the  $\psi_7$ edge at leaf $7$ as having weight $1$ .  These $\psi$ edges do not split; the other edges do.
    }
    \label{fig:psi_tree}
\end{figure}

\section{Acknowledgments}

We thank Vance Blankers, Renzo Cavalieri and two referees for helpful discussions pertaining to this work.

\bibliographystyle{alphaurl}
\bibliography{myrefs}

\end{document}